\documentclass[a4paper,11pt]{amsart}

\usepackage{graphicx}
\usepackage{amsmath}
\usepackage{amssymb}

\newtheorem{thm}{Theorem}
\newtheorem{lem}{Lemma}%
\theoremstyle{definition}

\theoremstyle{remark}
\newtheorem{remark}{Remark}[section] %

\theoremstyle{plain}

\numberwithin{equation}{section}

\def\NN{{\mathbb N}}

\def\RR{{\mathbb R}}
\def\TT{{\mathbb T}}

\def\ZZ{{\mathbb Z}}

\def\hatZZ{\widehat{\mathbb Z}}

\def\veca{{\text{\boldmath$a$}}}
\def\uveca{\widehat{\text{\boldmath$a$}}}
\def\cveca{\mathring{\text{\boldmath$a$}}}
\def\vecb{{\text{\boldmath$b$}}}

\def\vece{{\text{\boldmath$e$}}}

\def\vecm{{\text{\boldmath$m$}}}
\def\vecn{{\text{\boldmath$n$}}}

\def\vecp{{\text{\boldmath$p$}}}
\def\vecr{{\text{\boldmath$r$}}}

\def\vecv{{\text{\boldmath$v$}}}

\def\vecx{{\text{\boldmath$x$}}}

\def\vecy{{\text{\boldmath$y$}}}
\def\cvecy{\mathring{\text{\boldmath$y$}}}

\def\vecalf{{\text{\boldmath$\alpha$}}}

\def\veczeta{{\text{\boldmath$\zeta$}}}

\def\vecxi{{\text{\boldmath$\xi$}}}

\def\vecnull{{\text{\boldmath$0$}}}

\def\scrA{{\mathcal A}}
\def\scrB{{\mathcal B}}
\def\scrC{{\mathcal C}}
\def\scrD{{\mathcal D}}

\def\scrF{{\mathcal F}}
\def\scrI{{\mathcal I}}
\def\scrH{{\mathcal H}}

\def\scrK{{\mathcal K}}
\def\scrL{{\mathcal L}}
\def\scrM{{\mathcal M}}

\def\scrP{{\mathcal P}}
\def\scrR{{\mathcal R}}

\def\fC{{\mathfrak C}}

\def\e{\mathrm{e}}

\def\diag{\operatorname{diag}}

\def\minplu{\operatorname{{\min}_+}}

\def\Gr{\operatorname{Gr}}
\def\GL{\operatorname{GL}}

\def\S{\operatorname{S{}}}

\def\SL{\operatorname{SL}}
\def\ASL{\operatorname{ASL}}
\def\SO{\operatorname{SO}}

\def\OOO{\operatorname{O{}}}

\def\tr{\operatorname{tr}}

\def\supp{\operatorname{supp}}

\def\vol{\operatorname{vol}}

\def\GamG{\Gamma\backslash G}
\def\GamGG{\Gamma_0\backslash G_0}
\def\GamH{\Gamma_H\backslash H}

\def\SLZ{\SL(d,\ZZ)}
\def\SLR{\SL(d,\RR)}

\def\trans{\,^\mathrm{t}\!}

\title{The asymptotic distribution of Frobenius numbers}
\author{Jens Marklof}
\address{School of Mathematics, University of Bristol,
Bristol BS8 1TW, U.K.\newline
\rule[0ex]{0ex}{0ex} \hspace{8pt}{\tt j.marklof@bristol.ac.uk}}
\date{18 February 2009/17 December 2009. To appear in Inventiones Mathematicae.}
\thanks{The author is supported by a Royal Society Wolfson Research Merit Award.}

\begin{document}

\begin{abstract}
The Frobenius number $F(\veca)$ of an integer vector $\veca$ with positive coprime coefficients is defined as the largest number that does not have a representation as a positive integer linear combination of the coefficients of $\veca$. We show that if $\veca$ is taken to be random in an expanding $d$-dimensional domain, then $F(\veca)$ has a limit distribution, which is given by the probability distribution for the covering radius of a certain simplex with respect to a $(d-1)$-dimensional random lattice. This result extends recent studies for $d=3$ by Arnold, Bourgain-Sinai and Shur-Sinai-Ustinov. The key features of our approach are (a) a novel interpretation of the Frobenius number in terms of the dynamics of a certain group action on the space of $d$-dimensional lattices, and (b) an equidistribution theorem for a multidimensional Farey sequence on closed horospheres. 
\end{abstract}

\maketitle

\section{Introduction \label{secIntro}}

Let us denote by $\hatZZ^d=\{ \veca=(a_1,\ldots,a_d)\in\ZZ^d : \gcd(a_1,\ldots,a_d)=1 \}$ the set of primitive lattice points, and by $\hatZZ_{\geq 2}^d$ the subset with coefficients $a_j\geq 2$. Given $\veca\in\hatZZ_{\geq 2}^d$, it is well known that any sufficiently large integer $N>0$ can be represented in the form
\begin{equation}\label{rep}
	N=\vecm\cdot\veca
\end{equation}
with $\vecm\in\ZZ_{\geq 0}^d$. Frobenius was interested in the largest integer $F(\veca)$ that fails to have a representation of this type. That is,
\begin{equation}
	F(\veca) = \max \ZZ \setminus\{ \vecm\cdot\veca>0 : \vecm\in\ZZ_{\geq 0}^d \}.
\end{equation}
We will refer to $F(\veca)$ as the {\em Frobenius number} of $\veca$.
In the case of two variables ($d=2$) Sylvester showed that
\begin{equation}\label{sylvester}
	F(\veca)=a_1 a_2-a_1-a_2 .
\end{equation}
No such explicit formulas are known in higher dimensions, cf.~\cite{Ramirez05}, \cite{Rodseth78}, \cite{Selmer78}. The present paper will discuss a new interpretation of the Frobenius number in terms of the dynamics of a certain flow $\Phi^t$ on the space of lattices $\GamG$, with $G:=\SLR$, $\Gamma:=\SLZ$. This dynamical interpretation is a key step in the proof of the following limit theorem on the asymptotic distribution of the Frobenius number $F(\veca)$, where $\veca$ is randomly selected from the set $T\scrD=\{ T\vecx : \vecx\in\scrD\}$, with $T$ large and $\scrD$ a fixed bounded subset of $\RR_{\geq 0}^d$.

\begin{thm}\label{mainThm}
Let $d\geq 3$. There exists a continuous non-increasing function $\Psi_d: \RR_{\geq 0}\to \RR_{\geq 0}$ with $\Psi_d(0)=1$, such that for any bounded set $\scrD\subset\RR_{\geq 0}^d$ with boundary of Lebesgue measure zero, and any $R\geq 0$, 
\begin{equation}\label{mainThmeq}
	\lim_{T\to\infty} \frac{1}{T^d} \#\bigg\{ \veca\in\hatZZ_{\geq 2}^d\cap T\scrD : \frac{F(\veca)}{(a_1\cdots a_d)^{1/(d-1)}} >R \bigg\}
	=\frac{\vol(\scrD)}{\zeta(d)}\; \Psi_d(R) .
\end{equation}
\end{thm}

Variants of Theorem \ref{mainThm} were previously known only in dimension $d=3$, cf.~\cite{Bourgain07}, \cite{Shur}; see also \cite{Arnold99}, \cite{Arnold07} for related studies and open conjectures, and \cite{Aliev08}, \cite{Bourgain07} for results in higher dimensions. The scaling of $F(\veca)$ used in Theorem \ref{mainThm} is consistent with numerical experiments \cite[Section 5]{Beihoffer05}.

We will furthermore establish that the limit distribution $\Psi_d(R)$ is given by the distribution of the covering radius of the simplex
\begin{equation}\label{simplex}
	\Delta= \big\{ \vecx\in\RR_{\geq 0}^{d-1} : \vecx\cdot\vece\leq 1 \big\},\qquad
	\vece:=(1,1,\ldots,1),
\end{equation}
with respect to a random lattice in $\RR^{d-1}$. Here, the {\em covering radius} (sometimes also called {\em inhomogeneous minimum}) of a set $K\subset\RR^{d-1}$ with respect to a lattice $\scrL\subset\RR^{d-1}$ is defined as the infimum of all $\rho> 0$ with the property that $\scrL+\rho K=\RR^{d-1}$. 

To state this result precisely, let $\ZZ^{d-1}A$ be a lattice in $\RR^{d-1}$ with $A\in G_0:=\SL(d-1,\RR)$. The {\em space of lattices} (of unit covolume) is $\GamGG$ with $\Gamma_0:=\SL(d-1,\ZZ)$. We denote by $\mu_0$ the unique $G_0$-right invariant probability measure on $\GamGG$; an explicit formula for $\mu_0$ is given in Section \ref{secFarey}. 

\begin{thm}\label{mainThm2}
Let $\rho(A)$ be the covering radius of the simplex $\Delta$ with respect to the lattice $\ZZ^{d-1}A$. Then 
\begin{equation}\label{notrivi}
	\Psi_d(R) = \mu_0 \big(\big\{ A\in\GamGG : \rho(A)> R \big\}\big).
\end{equation}
\end{thm}

The connection between Frobenius numbers and lattice free simplices is well understood \cite{Kannan92}, \cite{Scarf93}. In particular, Theorem \ref{mainThm2} connects nicely to the sharp lower bound of \cite{Aliev07} (see also \cite{Rodseth90}):
\begin{equation}
	\frac{F(\veca)}{(a_1\cdots a_d)^{1/(d-1)}} \geq \rho_*,\qquad \text{with }
	\rho_* := \inf_{A\in\GamGG} \rho(A) .
\end{equation}
It is proved in \cite{Aliev07} that $\rho_*>((d-1)!)^{1/(d-1)}>0$, and so in particular
\begin{equation}\label{constant}
\Psi_d(R)=1 \qquad \text{for} \quad 0\leq R < \rho_*.	
\end{equation}
An explicit formula for $\Psi_d(R)$ has recently been derived in dimension $d=3$ by different techniques, cf.~\cite{Shur}. In this case $\rho_*=\sqrt{3}$.

It is amusing to note that all of the above statements also hold in the trivial case $d=2$, except for the continuity of the limit distribution:
By Sylvester's formula \eqref{sylvester}
\begin{equation}\label{trivi}
\Psi_2(R)=
\begin{cases}
1 & (R<1) \\
0 & (R\geq 1) .
\end{cases}	
\end{equation}
The covering radius of the simplex $\Delta=[0,1]$ with respect to the lattice $\ZZ$ is $\rho(1)=1$. $\ZZ$ is of course the unique element in the space of one-dimensional lattices of unit covolume, and hence \eqref{trivi} follows also formally from \eqref{notrivi}.

We now give a brief outline of the paper. Section \ref{secDyn} explains the aforementioned dynamical interpretation of the Frobenius number in terms of the right action of a one-parameter subgroup $\Phi^t$ on the space of lattices $\GamG$: We show that there is a function $W_\delta$ of $\GamG$ that produces, when evaluated along a certain orbit of $\Phi^t$, the Frobenius number $F(\veca)$. 
This observation is the crucial step in the application of an equidistribution theorem for multidimensional Farey sequences on closed horospheres in
$\GamG$, which is proved in Section \ref{secFarey}. A useful variant of this theorem is discussed in Section \ref{secVariant}. Section \ref{secUpp} exploits the equidistribution theorem to give upper and lower bounds for the lim sup and lim inf of \eqref{mainThmeq}, respectively, and the purpose of the remaining Sections \ref{secLat} and \ref{secCon} is to show that the lim sup and lim inf coincide. This is achieved by relating the limit distribution $\Psi_d(R)$ to the covering radius of a simplex with respect to a random lattice (Section \ref{secLat}), and proving that $\Psi_d(R)$ is continuous (Section \ref{secCon}).

The results of Sections \ref{secFarey} and \ref{secVariant} provide a new approach to Schmidt's work \cite{Schmidt98} on the distribution of (primitive) sublattices of $\ZZ^d$. Appendix \ref{AppA} illuminates this connection by deriving a generalization of Schmidt's Theorem 3 in the case of primitive sublattices of rank $d-1$.

\section{Dynamical interpretation \label{secDyn}}

Let $G:=\SLR$ and $\Gamma:=\SLZ$, and define
\begin{equation}
	n_+(\vecx)=\begin{pmatrix} 1_{d-1} & \trans\vecnull \\ \vecx & 1 \end{pmatrix},\qquad
	n_-(\vecx)=\begin{pmatrix} 1_{d-1} & \trans\vecx \\ \vecnull & 1 \end{pmatrix},\qquad \Phi^t=\begin{pmatrix} \e^{-t} 1_{d-1} & \trans\vecnull \\ \vecnull & \e^{(d-1)t} \end{pmatrix}.
\end{equation}
The right action
\begin{equation}
	\GamG \to \GamG, \qquad \Gamma M \mapsto \Gamma M \Phi^t
\end{equation}
defines a flow on the space of lattices $\GamG$.  
The horospherical subgroups generated by $n_+(\vecx)$ and $n_-(\vecx)$ parametrize the stable and unstable directions of the flow $\Phi^t$ as $t\to\infty$. This can be seen as follows. Let $d: G\times G\to \RR_{\geq 0}$ be a left $G$-invariant Riemannian metric on $G$, i.e., $d(h M, hM')=d(M,M')$ for all $h,M,M'\in G$. We may choose $d$ in such a way that 
\begin{equation}
	d\big(n_\pm(\vecx),n_\pm(\vecx')\big)\leq \| \vecx-\vecx'\| ,
\end{equation}
where $\|\,\cdot\,\|$ the standard euclidean norm.
Note that $n_-(\vecx)\Phi^t=\Phi^t n_-(\e^{dt}\vecx)$. Hence, for any $M\in G$,
\begin{equation}\label{expon}
	d\big(M n_-(\vecx)\Phi^t,M \Phi^t\big)
	= d\big(M \Phi^t n_-(\e^{dt} \vecx),M \Phi^t\big)
	= d\big(n_-(\e^{dt} \vecx), 1_d \big)
	\leq \e^{dt} \| \vecx \| ,
\end{equation}
which explains the interpretation of $n_-(\vecx)$ as an element in the {\em unstable} horospherical subgroup. The argument for $n_+(\vecx)$ as the stable analogue is identical. 

In the following we will represent functions on $\GamG$ as left $\Gamma$-invariant functions on $G$, i.e., functions $f:G\to\RR$ that satisfy $f(\gamma M)=f(M)$ for all $\gamma\in\Gamma$.
The left $G$-invariant metric $d(\,\cdot\,,\,\cdot\,)$ yields thus a Riemannian metric $d_\Gamma(\,\cdot\,,\,\cdot\,)$ on $\GamG$ by setting
\begin{equation}\label{m1}
	d_\Gamma(M,M'):=\min_{\gamma\in\Gamma} d(M,\gamma M') .
\end{equation}
Indeed, the left $G$-invariance of $d$ implies $d_\Gamma(\gamma M,M')=d_\Gamma(M,M')=d_\Gamma(M,\gamma M')$ for any $\gamma\in\Gamma$.

The aim of the present section is to identify a function $W_\delta$ on $\GamG$ that, when evaluated along a specific orbit of the flow $\Phi^t$, produces the Frobenius number. (As we shall see below, the situation is slightly more complicated in that $W_\delta$ also depends on additional variables in $\RR^{d-1}$.)  

We will assume throughout that $\veca\in\hatZZ_{\geq 2}^d$. Following  \cite{Brauer62}, \cite{Selmer77} we reduce the Frobenius problem modulo $a_d$. For $r\in\ZZ/a_d\ZZ$ set
\begin{equation}
	F_r(\veca) = \max (r+a_d\ZZ)\setminus\{ \vecm\cdot\veca>0 : \vecm\in\ZZ_{\geq 0}^d,\; \vecm\cdot\veca\equiv r \bmod a_d \}
\end{equation}
Then
\begin{equation}\label{Fmax}
	F(\veca)=\max_{r\bmod a_d} F_r(\veca).
\end{equation}
Consider the smallest positive integer that has a representation in $r\bmod a_d$,
\begin{equation}
	N_r(\veca) = \min\{ \vecm\cdot\veca>0 : \vecm\in\ZZ_{\geq 0}^{d},\;\vecm\cdot\veca\equiv r \bmod a_d \} . 
\end{equation}
Then $F_r(\veca)=N_r(\veca)-a_d$. We have in fact
\begin{equation}
	N_r(\veca) =
	\begin{cases}
	a_d & (r\equiv 0\bmod a_d) \\ 
	 \min\{ \vecm'\cdot\veca': \vecm'\in\ZZ_{\geq 0}^{d-1},\;\vecm'\cdot\veca'\equiv r \bmod a_d \} & (r\not\equiv 0\bmod a_d) 
	 \end{cases}
\end{equation}
with $\veca'=(a_1,\ldots,a_{d-1})$. In view of \eqref{Fmax} we conclude
\begin{equation}\label{keyEq}
	F(\veca)=\max_{r\not\equiv 0 \bmod a_d} N_r(\veca)  -a_d.
\end{equation}

We assume in the following $a_1,\ldots,a_{d-1}\leq a_d\leq T$, and $0<\delta\leq \frac12$. For $r\not\equiv 0 \bmod a_d$ we then have
\begin{equation}
	N_r(\veca)  = \min\bigg\{ \vecm' \cdot \veca' :
	\vecm\in\ZZ_{\geq 0}^{d-1}\times\ZZ,\; \big|\vecm\cdot \veca - r \big| < \frac{\delta a_d}{T} \bigg\} .
\end{equation}
For $\vecxi=(\vecxi',\xi_d)\in\TT^d=\RR^d/\ZZ^d$, set
\begin{equation}
	N(\veca,\vecxi,T):=\minplu\bigg\{(\vecm'+\vecxi') \cdot\veca' :
	\vecm+\vecxi\in(\ZZ^d+\vecxi)\cap\RR_{\geq 0}^{d-1}\times\RR ,\; \big|(\vecm+\vecxi)\cdot \veca  \big| < \frac{\delta a_d}{T} \bigg\},
\end{equation}
where $\minplu$ is defined by 
\begin{equation}
	\minplu \scrA = \begin{cases} \min \scrA\cap\RR_{\geq 0} & (\scrA\cap\RR_{\geq 0}\neq\emptyset)\\ 0 & (\scrA\cap\RR_{\geq 0}=\emptyset) . \end{cases}
\end{equation}
It is evident that $N(\veca,\vecxi,T)$ is indeed well defined as a function of $\vecxi\in\TT^d$, and furthermore $N_r(\veca)=N(\veca,(\vecnull,-\frac{r}{a_d}),T)$.

\begin{lem}\label{lem1}
Let $\veca=(a_1,\ldots,a_d)\in\hatZZ_{\geq 2}^d$ with $a_1,\ldots,a_{d-1}\leq a_d\leq T$, $0<\delta\leq\frac12$. Then
\begin{equation}
	F(\veca)
	=\sup_{\vecxi\in\RR^d/\ZZ^d}  N(\veca,\vecxi,T) -\vece\cdot\veca  ,
\end{equation}
where $\vece=(1,1,\ldots,1)$.
\end{lem}

\begin{proof}
Substituting $\xi_d$ by $\xi_d-\vecxi'\cdot\frac{\veca'}{a_d}$, we have
\begin{equation}\label{NNAA}
\begin{split}
&\sup_{\vecxi\in\RR^d/\ZZ^d}  N(\veca,\vecxi,T) \\
& = \sup_{\substack{\vecxi'\in[0,1)^{d-1}\\ \xi_d\in\TT^1}}
\minplu\bigg\{(\vecm'+\vecxi') \cdot\veca' :
	\vecm+\vecxi\in(\ZZ^d+\vecxi)\cap\RR_{\geq 0}^{d-1}\times\RR ,\; \big|\vecm\cdot \veca +\xi_d a_d \big| < \frac{\delta a_d}{T} \bigg\}\\
& = \sup_{\substack{\vecxi'\in[0,1)^{d-1}\\ \xi_d\in\TT^1}}
\minplu\bigg\{(\vecm'+\vecxi') \cdot\veca' :
	\vecm\in\ZZ_{\geq 0}^{d-1}\times\ZZ,\; \big|\vecm\cdot \veca +\xi_d a_d \big| < \frac{\delta a_d}{T} \bigg\} \\
	& =  \sup_{\xi_d\in\TT^1}
\minplu\bigg\{\vecm' \cdot\veca' :
	\vecm\in\ZZ_{\geq 0}^{d-1}\times\ZZ,\; \big|\vecm\cdot \veca +\xi_d a_d \big| < \frac{\delta a_d}{T} \bigg\} + \vece\cdot\veca'  ,
	\end{split}
\end{equation}
where $\vece=(1,1,\ldots,1)$.
The second equality follows from the fact that for $1\leq j<d$, $m_j+\xi_j\geq 0$ implies $m_j\geq 0$ since $m_j\in\ZZ$ and $\xi_j\in[0,1)$. We observe that, since $\frac{\delta a_d}{T}\leq \frac12$ and $\vecm\cdot\veca\in\ZZ$, we can replace in the inequality 
$|\vecm\cdot \veca +\xi_d a_d | < \frac{\delta a_d}{T}$ the quantity $\xi_d a_d$ by its nearest integer, say $s$. That is, \eqref{NNAA} equals
\begin{equation}
\sup_{s \bmod a_d}
\minplu\bigg\{\vecm' \cdot\veca' :
	\vecm\in\ZZ_{\geq 0}^{d-1}\times\ZZ,\; \big|\vecm\cdot \veca + s \big| < \frac{\delta a_d}{T} \bigg\} + \vece\cdot\veca' .
\end{equation}
The case $s\equiv 0\bmod a_d$ does not contribute (because then $\vecm=\vecnull$ achieves 0 as minimum). Since $0\leq a_j\leq a_d$ we thus obtain
\begin{equation}
	\max_{r\not\equiv 0 \bmod a_d} N_r(\veca)
	=\sup_{\vecxi\in\RR^d/\ZZ^d}  N(\veca,\vecxi,T) - \vece\cdot\veca',
\end{equation}
and the lemma follows from \eqref{keyEq}.
\end{proof}

Let $W_\delta$ denote the function $\RR_{\geq 0}^{d-1}\times G\to\RR$, $(\vecalf,M)\mapsto W_\delta(\vecalf,M)$, given by
\begin{equation}\label{fdef}
	W_\delta(\vecalf,M)= \sup_{\vecxi\in\TT^d} \minplu\big\{ (\vecm+\vecxi)M \cdot(\vecalf,0) :
	\vecm\in\ZZ^{d},\; (\vecm+\vecxi)M\in \scrR_\delta \big\} 
\end{equation}
where $\scrR_\delta=\RR_{\geq 0}^{d-1}\times (-\delta,\delta)$.
Note that for every $\gamma\in\Gamma$
\begin{equation} \label{Winv}
\begin{split}
	W_\delta(\vecalf,\gamma M) & = \sup_{\vecxi\in\TT^d} \minplu\big\{ (\vecm+\vecxi)\gamma M \cdot(\vecalf,0) :
	\vecm\in\ZZ^{d},\; (\vecm+\vecxi)\gamma M\in \scrR_\delta \big\} \\
	& = \sup_{\vecxi\in \TT^d\gamma} \minplu\big\{ (\vecm+\vecxi)M \cdot(\vecalf,0) :
	\vecm\in\ZZ^{d}\gamma ,\; (\vecm+\vecxi)M\in \scrR_\delta \big\} .
\end{split}
\end{equation}
Both $\ZZ^d$ and $\TT^d$ are $\Gamma$-invariant; thus
\begin{equation}\label{leftW}
	W_\delta(\vecalf,\gamma M)=W_\delta(\vecalf,M)
\end{equation}
for all $\vecalf\in\RR_{\geq 0}^{d-1}$, $M\in G$ and $\gamma\in\Gamma$. 

Combining Definition \eqref{fdef} with Lemma \ref{lem1} (set $t=\frac{\log T}{d-1}$) we obtain:

\begin{thm}\label{WThm}
Let $\veca=(a_1,\ldots,a_d)\in\hatZZ_{\geq 2}^d$ with $a_1,\ldots,a_{d-1}\leq a_d\leq \e^{(d-1)t}$, and $0<\delta\leq\frac12$. Then
\begin{equation}
	F(\veca)
	= \e^{t} W_\delta\big(\veca',n_-(\uveca)\Phi^t\big) - \vece\cdot\veca ,
\end{equation}
where
\begin{equation}\label{uveca}	\uveca=\frac{\veca'}{a_d}=\bigg(\frac{a_1}{a_{d}},\ldots,\frac{a_{d-1}}{a_{d}}\bigg).
\end{equation}
\end{thm}

\section{Farey sequences on horospheres\label{secFarey}}

Denote by
$\mu=\mu_{G}$ the Haar measure on $G=\SLR$, normalized so that it represents the unique right $G$-invariant probability measure on the homogeneous space $\GamG$, where $\Gamma=\SLZ$. By Siegel's volume formula
\begin{equation} \label{siegel}
d\mu(M) \frac{dt}t = \big(\zeta(2)\zeta(3)\cdots\zeta(d)\big)^{-1}\;\det (X)^{-d}
\prod_{i,j=1}^d dX_{ij},
\end{equation}
where $X=(X_{ij})=t^{1/d}M \in \GL^+(d,\RR)$ with $M\in G$, $t>0$, cf.~\cite{Marklof00}, \cite{siegel}. 
We will also use the notation $\mu_0$ for the right $G_0$-invariant probability measure on $\GamGG$, with $G_0=\SL(d-1,\RR)$ and $\Gamma_0=\SL(d-1,\ZZ)$.

Consider the subgroups
\begin{equation}\label{Hdef}
	H = \bigg\{ M\in G :  (\vecnull,1)M =(\vecnull,1) \bigg\}= \bigg\{ \begin{pmatrix} A  & \trans\vecb \\ \vecnull & 1 \end{pmatrix} : A\in G_0,\; \vecb\in\RR^{d-1} \bigg\}
\end{equation}
and
\begin{equation}
	\Gamma_H = \Gamma\cap H  = \bigg\{ \begin{pmatrix} \gamma & \trans\vecm \\ \vecnull & 1 \end{pmatrix} : \gamma\in\Gamma_0,\; \vecm\in\ZZ^{d-1} \bigg\} .
\end{equation}
Note that $H$ and $\Gamma_H$ are isomorphic to $\ASL(d-1,\RR)$ and $\ASL(d-1,\ZZ)$, respectively. We normalize the Haar measure $\mu_H$ of $H$ so that it becomes a probability measure on $\GamH$; explicitly:
\begin{equation} \label{siegel2}
d\mu_H(M) = d\mu_0(A)\, d\vecb, \qquad M=\begin{pmatrix} A & \trans\vecb \\ \vecnull & 1 \end{pmatrix}.
\end{equation}

The following states the classical equidistribution theorem for $\Phi^{t}$-translates of the closed horospheres $\Gamma\backslash\Gamma\{ n_-(\vecx) : \vecx\in\TT^{d-1} \}$ on $\GamG$; cf.~\cite[Section 5]{partI}.

\begin{thm}\label{equiThm0}
Let $\lambda$ be a Borel probability measure on $\TT^{d-1}$, absolutely continuous with respect to Lebesgue measure, and let $f:\TT^{d-1}\times\GamG\to\RR$ be bounded continuous. Then
\begin{equation}
	\lim_{t\to\infty} \int_{\TT^{d-1}} f\big(\vecx,n_-(\vecx)\Phi^{t}\big)\, d\lambda(\vecx)  = \int_{\TT^{d-1}\times\GamG} f(\vecx,M) \, d\lambda(\vecx) \, d\mu(M) .
\end{equation}
\end{thm}

A standard probabilistic argument \cite[Chapter III]{Shiryaev} allows to reformulate the above statement in terms characteristic functions of subsets of $\TT^{d-1}\times\GamG$. 

\begin{thm}\label{equiThm0b}
Take $\lambda$ as in Theorem \ref{equiThm0}, and let $\scrA\subset\TT^{d-1}\times\GamG$. Then
\begin{equation}
	\liminf_{t\to\infty} \lambda\big(\big\{ \vecx\in\TT^{d-1}: \big(\vecx,n_-(\vecx)\Phi^{t}\big)\in\scrA \big\}\big)  \geq 
	(\lambda\times\mu)(\scrA^\circ) 
\end{equation}
and
\begin{equation}
	\limsup_{t\to\infty} \lambda\big(\big\{ \vecx\in\TT^{d-1}: \big(\vecx,n_-(\vecx)\Phi^{t}\big)\in\scrA \big\}\big)  \leq 
	(\lambda\times\mu)(\overline\scrA) .
\end{equation}
\end{thm}

\begin{remark}\label{remA}
This shows that Theorem \ref{equiThm0} can be extended to test functions $f$ that are characteristic functions of subsets of $\TT^{d-1}\times\GamG$ with boundary of $(\lambda\times\mu)$-measure zero \cite[Sect.~5.3]{partI}, and thus also to functions that are the product of such a characteristic function and a bounded continuous function.
\end{remark}

We will now replace the absolutely continuous measure $\lambda$ by equally weighted point masses at the elements of the Farey sequence
\begin{equation}\label{fareydef}
	\scrF_Q=\bigg\{ \frac{\vecp}{q} \in[0,1)^{d-1} : (\vecp,q)\in\hatZZ^d, \; 0<q\leq Q \bigg\} ,
\end{equation}
for $Q\in\NN$.
Note that
\begin{equation}\label{asymQ}
	|\scrF_Q| \sim \frac{Q^d}{d\,\zeta(d)} \qquad (Q\to\infty) .
\end{equation}
It will be notationally convenient to also allow general $Q\in\RR_{\geq 1}$ in the definition \eqref{fareydef} of $\scrF_Q$; note that $\scrF_Q=\scrF_{[Q]}$ where $[Q]$ is the integer part of $Q$.

\begin{thm}\label{equiThm1}
Fix $\sigma\in\RR$. Let $f:\TT^{d-1}\times\GamG\to\RR$ be bounded continuous. Then, for $Q=\e^{(d-1)(t-\sigma)}$,
\begin{multline}\label{equiThm1eq}
	\lim_{t\to\infty} \frac{1}{|\scrF_Q|} \sum_{\vecr\in\scrF_Q} f\big(\vecr,n_-(\vecr)\Phi^{t}\big)  \\
	= d(d-1)\e^{d(d-1)\sigma}\int_{\sigma}^\infty \int_{\TT^{d-1}\times\GamH} \widetilde f(\vecx, M \Phi^{-s}) \, d\vecx \, d\mu_H(M) \, \e^{-d(d-1)s} ds 
\end{multline}
with $\widetilde f(\vecx,M):=f(\vecx,\trans M^{-1})$.
\end{thm}

\begin{remark}\label{remB}
The identical argument as in Remark \ref{remA} permits the extension of Theorem \ref{equiThm1} to any test function $f$ which is the product of a bounded continuous function and the characteristic function of a subset $\scrA\subset \TT^{d-1}\times\GamG$, where $\widetilde\scrA=\{(\vecx,M): (\vecx,\trans M^{-1})\in\scrA\}$ has boundary of measure zero with respect to $d\vecx \, d\mu_H \, ds$.
\end{remark}

\begin{proof}[Proof of Theorem \ref{equiThm1}]
{\bf Step 0: Uniform continuity.} By choosing the test function $f(\vecx,M)=f_0(\vecx,M\Phi^{-\sigma})$ with $f_0:\TT^{d-1}\times\GamG\to\RR$ bounded continuous, it is evident that we only need to consider the case $\sigma=0$. We may also assume without loss of generality that $f$, and thus $\widetilde f$, have compact support. That is, there is $\scrC\subset G$ compact such that $\supp f,\supp\widetilde f\subset\TT^{d-1}\times \Gamma\backslash\Gamma\scrC$. The generalization to bounded continuous functions follows from a standard approximation argument.

Since $f$ is continuous and has compact support, it is uniformly continuous. That is, given any $\delta>0$ there exists $\epsilon>0$ such that for all $(\vecx,M),(\vecx',M')\in \RR^{d-1}\times G$,
\begin{equation}\label{epsi}
	\|\vecx-\vecx'\| < \epsilon, \qquad d(M,M') < \epsilon 
\end{equation}
implies
\begin{equation}\label{delt}
	\big| f(\vecx,M) - f(\vecx',M') \big| < \delta .
\end{equation}

The plan is now to first establish \eqref{equiThm1eq} for the set
\begin{equation}
	\scrF_{Q,\theta}=\bigg\{ \frac{\vecp}{q} \in[0,1)^{d-1} : (\vecp,q)\in\hatZZ^d, \; \theta Q<q\leq Q \bigg\} ,
\end{equation}
for any $\theta\in(0,1)$. The constant $\theta$ will remain fixed until the very last step of this proof.

{\bf Step 1: Thicken the Farey sequence.} The plan is to reduce the statement to Theorem \ref{equiThm0}. To this end, we thicken the set $\scrF_{Q,\theta}$ as follows: For $\epsilon>0$ (we will in fact later use the $\epsilon$ from Step 0), let
\begin{equation}\label{FQeps0}
	\scrF_Q^\epsilon =  \bigcup_{\vecr\in\scrF_{Q,\theta}+\ZZ^{d-1}} \big\{ \vecx\in \RR^{d-1} : \|\vecx-\vecr\| < \epsilon\e^{-dt} \big\}  .
\end{equation}
Note that $\scrF_Q^\epsilon$ is symmetric with respect to $\vecx\mapsto-\vecx$.
A short calculation yields
\begin{equation}\label{FQeps}
	\scrF_Q^\epsilon = \bigcup_{\veca\in\hatZZ^d} \big\{ \vecx\in \RR^{d-1} :  \veca\,n_+(\vecx)\Phi^{-t}\in \fC_\epsilon \big\} ,
\end{equation}
where
\begin{equation}\label{Ceps}
	\fC_\epsilon= \big\{ (y_1,\ldots,y_d) \in\RR^d : \| (y_1,\ldots,y_{d-1}) \| < \epsilon y_d , \; \theta <y_d\leq 1 \big\} .
\end{equation}
Let
\begin{equation}\label{Ha}
	\scrH_\epsilon=\bigcup_{\veca\in\hatZZ^d} \scrH_\epsilon(\veca),\qquad 
	\scrH_\epsilon(\veca) = \big\{ M\in G:\veca M \in \fC_\epsilon \big\} .
\end{equation}
The bijection (cf.~\cite{siegel})
\begin{equation}\label{bij}
	\Gamma_H\backslash\Gamma \to \hatZZ^d, \qquad \Gamma_H \gamma \mapsto (\vecnull,1) \gamma 
\end{equation}
allows us to rewrite
\begin{equation}\label{HaHa}
	\scrH_\epsilon=\bigcup_{\gamma\in\Gamma_H\backslash\Gamma} \scrH_\epsilon((\vecnull,1)\gamma)=\bigcup_{\gamma\in\Gamma/\Gamma_H} \gamma\scrH_\epsilon^1 ,
	\qquad\text{with } \scrH_\epsilon^1= \scrH_\epsilon((\vecnull,1)).
\end{equation}
Now
\begin{equation}\label{MyMy}
\begin{split}
	\scrH_\epsilon^1 = & \big\{ M\in G: (\vecnull,1) M \in \fC_\epsilon \big\} \\
	= & H \big\{ M_\vecy : \vecy \in \fC_\epsilon \big\}
\end{split}
\end{equation}
with $H$ as in \eqref{Hdef}, and $M_\vecy\in G$ such that $(\vecnull,1) M_\vecy=\vecy$. Since $\vecy\in\fC_\epsilon$ implies $y_d>0$, we may choose
\begin{equation}\label{My}
	M_\vecy=\begin{pmatrix} y_d^{-1/(d-1)} 1_{d-1} & \trans\vecnull \\ \vecy' & y_d\end{pmatrix}, \qquad \vecy'=(y_1,\ldots,y_{d-1}).
\end{equation}

{\bf Step 2: Prove disjointness.} We will now prove the following claim: {\em Given a compact subset $\scrC\subset G$, there exists $\epsilon_0>0$ such that 
\begin{equation}\label{C1}
	\gamma\scrH_\epsilon^1\cap \scrH_\epsilon^1\cap\Gamma\scrC=\emptyset
\end{equation}
for every $\epsilon\in(0,\epsilon_0]$, $\gamma\in \Gamma\setminus\Gamma_H$.} 

To prove this claim, note that \eqref{C1} is equivalent to 
\begin{equation}\label{C2}
	\scrH_\epsilon((\vecp,q))\cap\scrH_\epsilon^1\cap\Gamma\scrC=\emptyset
\end{equation}
for every $(\vecp,q)\in\hatZZ^d$, $(\vecp,q)\neq(\vecnull,1)$. For 
\begin{equation}\label{MAb}
	M=\begin{pmatrix} A & \trans\vecb \\ \vecnull & 1 \end{pmatrix} M_\vecy, \qquad
	M_\vecy=\begin{pmatrix} y_d^{-1/(d-1)} 1_{d-1} & \trans\vecnull \\ \vecy' & y_d\end{pmatrix},
\end{equation}
we have
\begin{equation}
	(\vecp,q) M
	=(\vecp A y_d^{-1/(d-1)}+(\vecp\trans\vecb+q)\vecy', (\vecp\trans\vecb+q)y_d),
\end{equation}
and thus $M\in\scrH_\epsilon((\vecp,q))\cap\scrH_\epsilon^1$ if and only if
\begin{equation}\label{A1}
	\|\vecp A y_d^{-1/(d-1)}+(\vecp\trans\vecb+q)\vecy'\|<\epsilon (\vecp\trans\vecb+q)y_d,
\end{equation}
\begin{equation}\label{A12}
	\theta <(\vecp\trans\vecb+q)y_d \leq 1,
\end{equation}
and 
\begin{equation}\label{A2}
	\| \vecy' \| < \epsilon y_d , \qquad \theta <y_d\leq 1 .
\end{equation}
Relations \eqref{A12} and \eqref{A2} imply $\|(\vecp\trans\vecb+q)\vecy'\|<\epsilon(\vecp\trans\vecb+q)y_d\leq \epsilon$ and so, by \eqref{A1}, $\|\vecp A y_d^{-1/(d-1)}\|<2\epsilon(\vecp\trans\vecb+q)y_d\leq 2\epsilon$. That is, $\|\vecp A \|<2\epsilon y_d^{1/(d-1)}$ and hence 
\begin{equation}\label{A2e}
\|\vecp A \|< 2\epsilon.	
\end{equation}
Let us now suppose $M\in\Gamma\scrC$ with $\scrC$ compact. The set
\begin{equation}
	\scrC'=\scrC \big\{ M_\vecy^{-1} : \vecy \in \overline \fC_\epsilon \big\} 
\end{equation}
is still compact, by the compactness of $\overline \fC_\epsilon$ (the closure of $\fC_\epsilon$) in $\RR^d\setminus\{\vecnull\}$. In view of \eqref{MAb} we obtain 
\begin{equation}
	\begin{pmatrix} A & \trans\vecb \\ \vecnull & 1 \end{pmatrix} \in \Gamma\scrC',
\end{equation}
and so $A\in \Gamma_0 \scrC_0$ for some compact $\scrC_0\subset G_0$.

Mahler's compactness criterion then shows that
\begin{equation}
	I:=\inf_{A\in\Gamma_0\scrC_0}\inf_{\vecp\in\ZZ^{d-1}\setminus\{\vecnull\}}\|\vecp A \|>0 .
\end{equation}
Now choose $\epsilon_0$ such that $0<2\epsilon_0<I$. Then \eqref{A2e} implies $\vecp=\vecnull$ and therefore $q=1$. The claim is proved.

{\bf Step 3: Apply Theorem \ref{equiThm0}.} Step 2 implies that, for $\scrC\subset G$ compact, there exists $\epsilon_0>0$ such that for every $\epsilon\in(0,\epsilon_0]$
\begin{equation}\label{HaHaHa}
	\scrH_\epsilon\cap \Gamma\scrC =\bigcup_{\gamma\in\Gamma/\Gamma_H} \big( \gamma\scrH_\epsilon^1 \cap \Gamma\scrC \big) 
\end{equation}
is a disjoint union. Hence, if $\chi_\epsilon$ and $\chi_\epsilon^1$ are the characteristic functions of the sets $\scrH_\epsilon$ and $\scrH_\epsilon^1$, respectively, we have
\begin{equation}\label{chi}
	\chi_\epsilon(M) = \sum_{\gamma\in\Gamma_H\backslash\Gamma} \chi_{\epsilon}^1(\gamma M) ,
\end{equation}
for all $M\in\Gamma\scrC$.
Evidently $\scrH_\epsilon^1$ and thus $\scrH_\epsilon$ have boundary of $\mu$-measure zero.
We furthermore set $\widetilde\chi_\epsilon(M):=\chi_\epsilon(\trans M^{-1})$, and note that  
$\chi_\epsilon\big(n_+(\vecx)\Phi^{-t}\big)=\chi_\epsilon\big(n_+(-\vecx)\Phi^{-t}\big)$ is the characteristic function of the set $\scrF_Q^\epsilon$; recall \eqref{FQeps} and the remark after \eqref{FQeps0}.
Therefore
\begin{equation}\label{commEq2}
\begin{split}
	\int_{\scrF_Q^\epsilon/\ZZ^{d-1}} f\big(\vecx,n_-(\vecx)\Phi^{t}\big) d\vecx & = \int_{\TT^{d-1}} f\big(\vecx,n_-(\vecx)\Phi^{t}\big) 
	\chi_\epsilon\big(n_+(-\vecx)\Phi^{-t}\big) d\vecx\\
	& = \int_{\TT^{d-1}} f\big(\vecx,n_-(\vecx)\Phi^{t}\big) 
	\widetilde\chi_\epsilon\big(n_-(\vecx)\Phi^{t}\big) d\vecx ,
\end{split}
\end{equation}
and Theorem \ref{equiThm0} yields
\begin{equation}\label{last}
\begin{split}
	\lim_{t\to\infty} \int_{\TT^{d-1}}  f\big(\vecx,n_-(\vecx)\Phi^{t}\big)\widetilde\chi_\epsilon\big(n_-(\vecx)\Phi^{t}\big)\, d\vecx  
	& =  \int_{\TT^{d-1}\times \GamG} f(\vecx,M) \widetilde\chi_\epsilon(M)\, d\vecx \, d\mu(M) \\
	& = \int_{\TT^{d-1}\times \GamG} \widetilde f(\vecx,M) \chi_\epsilon(M)\, d\vecx \, d\mu(M) .
\end{split}
\end{equation}

{\bf Step 4: A volume computation.}
To evaluate the right hand side of \eqref{last}, we use \eqref{chi}:
\begin{equation}\label{RhS}
\begin{split}
	 \int_{\TT^{d-1}\times \GamG} \widetilde f(\vecx,M) \chi_\epsilon(M)\, d\vecx \, d\mu(M)
	 & = \int_{\TT^{d-1}\times \Gamma_H\backslash G} \widetilde f(\vecx,M) \chi_\epsilon^1(M)\, d\vecx \, d\mu(M)\\
	 & = \int_{\TT^{d-1}\times \Gamma_H\backslash\scrH_\epsilon^1} \widetilde f(\vecx,M)\, d\vecx \, d\mu(M) .
\end{split}
\end{equation}
Given $\vecy\in\RR^d$ we pick a matrix $M_\vecy\in G$ such that $(\vecnull,1) M_\vecy=\vecy$; recall \eqref{My} for an explicit choice of $M_\vecy$ for $y_d>0$. The map
\begin{equation}
	H \times \RR^d\setminus\{\vecnull\} \to  G, \qquad 
	(M, \vecy) \mapsto MM_\vecy, 
\end{equation}
provides a parametrization of $G$, where in view of \eqref{siegel}
\begin{equation}
	d\mu = \zeta(d)^{-1} d\mu_H \, d\vecy .
\end{equation}
Hence \eqref{RhS} equals
\begin{equation}\label{RhS1}
	\frac{1}{\zeta(d)} \int_{\TT^{d-1}\times \Gamma_H\backslash H \times\fC_\epsilon} \widetilde f\big(\vecx,M M_\vecy \big) \, d\vecx \, d\mu_H(M) d\vecy .
\end{equation}

For
\begin{equation}\label{Ddef}
	D(y_d)=\begin{pmatrix} y_d^{-1/(d-1)} 1_{d-1} & \trans\vecnull \\ \vecnull & y_d \end{pmatrix},
\end{equation}
we have
\begin{equation}
	d\big(M_\vecy, D(y_d)\big)= d\big(D(y_d) n_+(y_d^{-1}\vecy'), D(y_d)\big)
	= d\big(n_+(y_d^{-1}\vecy'), 1_d \big)
	\leq y_d^{-1} \| \vecy' \| . 
\end{equation}
We recall that $y_d^{-1} \| \vecy' \|<\epsilon$ for $\vecy\in\fC_\epsilon$.
Therefore, with the choice of $\delta,\epsilon$ made in Steps 0 and 2, we have (note that \eqref{delt} applies also to $\widetilde f$)
\begin{equation}\label{f0}
	\bigg| \eqref{RhS1} -
	\frac{1}{\zeta(d)} \int_{\TT^{d-1}\times \Gamma_H\backslash H\times \fC_\epsilon} \widetilde f\big(\vecx,M D(y_d)\big) \, d\vecx \, d\mu_H(M) d\vecy \bigg|  <   \frac{\delta}{\zeta(d)}\int_{\fC_\epsilon} d\vecy .
\end{equation}

We have
\begin{equation}
\begin{split}
	\int_{\fC_\epsilon} \widetilde f\big(\vecx,M D(y_d) \big) \, d\vecy
	& = \vol(\scrB_1^{d-1})\,\epsilon^{d-1} \int_\theta^1 \widetilde f\big(\vecx,M D(y_d)\big) \, y_d^{d-1}\,dy_d \\
	& = (d-1) \vol(\scrB_1^{d-1})\,\epsilon^{d-1} \int_{0}^{|\log\theta|/(d-1)} \widetilde f\big(\vecx,M\Phi^{-s} \big) \,\e^{-d(d-1)s} ds ,
	\end{split} 
\end{equation}
and
\begin{equation}
	\int_{\fC_\epsilon} d\vecy = \frac{1}{d}\,\vol(\scrB_1^{d-1})\,\epsilon^{d-1} (1-\theta^d),
\end{equation}
where $\scrB_1^{d-1}$ denotes the unit ball in $\RR^{d-1}$.
So \eqref{f0} becomes
\begin{multline}\label{f1}
	\bigg| \eqref{RhS1} -
	\frac{(d-1) \vol(\scrB_1^{d-1})\,\epsilon^{d-1}}{\zeta(d)} \int_{0}^{|\log\theta|/(d-1)} \int_{\TT^{d-1}\times \Gamma_H\backslash H}  \widetilde f\big(\vecx,M\Phi^{-s} \big)   \, d\vecx \, d\mu_H(M)\, \e^{-d(d-1)s} ds \bigg| \\ <  \frac{\vol(\scrB_1^{d-1})\,\delta \, \epsilon^{d-1}}{d\,\zeta(d)} \,(1-\theta^d)  .
\end{multline}

{\bf Step 5: Distance estimates.} Since \eqref{HaHaHa} is a disjoint union, we have furthermore (this is in effect another way of writing \eqref{commEq2} using \eqref{chi})
\begin{equation}
\int_{\scrF_Q^\epsilon/\ZZ^{d-1}} f\big(\vecx,n_-(\vecx)\Phi^{t}\big) d\vecx
=	\sum_{\vecr\in\scrF_{Q,\theta}} \int_{\|\vecx-\vecr\|<\epsilon\e^{-dt}} f\big(\vecx,n_-(\vecx)\Phi^{t}\big) d\vecx .
\end{equation}
Eq.~\eqref{expon} implies that 
\begin{equation}
	d\big(n_-(\vecx)\Phi^t,n_-(\vecr) \Phi^t\big)
	\leq \e^{dt} \| \vecx -\vecr \| <\epsilon .
\end{equation}
Because $f$ is uniformly continuous we therefore have, for the same $\delta,\epsilon$ as above:
\begin{equation}\label{f2}
\bigg| \int_{\|\vecx-\vecr\|<\epsilon\e^{-dt}} f\big(\vecx,n_-(\vecx)\Phi^{t}\big) d\vecx - \frac{\vol(\scrB_1^{d-1}) \epsilon^{d-1}}{\e^{d(d-1)t}} f\big(\vecr,n_-(\vecr)\Phi^{t}\big)  \bigg| < \frac{\vol(\scrB_1^{d-1}) \,\delta\, \epsilon^{d-1}}{\e^{d(d-1)t}},
\end{equation}
uniformly for all $t\geq 0$.

{\bf Step 6: Conclusion.}
The approximations \eqref{f1} and \eqref{f2} hold uniformly for any $\delta>0$. Passing to the limit $\delta\to 0$, we obtain
\begin{multline}\label{lalalast}
	\lim_{t\to\infty} \frac{1}{\e^{d(d-1)t}} \sum_{\vecr\in\scrF_{Q,\theta}} f\big(\vecr,n_-(\vecr)\Phi^{t}\big) \\
	= \frac{d-1}{\zeta(d)} \int_{0}^{|\log\theta|/(d-1)} \int_{\TT^{d-1}\times \Gamma_H\backslash H}  \widetilde f\big(\vecx,M\Phi^{-s} \big)   \, d\vecx \, d\mu_H(M)\, \e^{-d(d-1)s} ds .
\end{multline}
The asymptotics \eqref{asymQ} show that
\begin{equation}
	\limsup_{t\to\infty} \frac{|\scrF_Q\setminus \scrF_{Q,\theta}|}{\e^{d(d-1)t}}  \leq \frac{\theta^d}{d\,\zeta(d)} ,
\end{equation}
which allows us to take the limit $\theta\to 0$ in \eqref{lalalast}.
This concludes the proof for $\sigma=0$ and $f$ compactly supported. For the general case, recall the remarks in Step 0.
\end{proof}

\begin{remark}\label{remy} 
Let $(\vecp,q)\in\hatZZ$. Using the bijection \eqref{bij}, choose $\gamma\in\Gamma$ such that $(\vecp,q)\gamma=(\vecnull,1)$. For $\vecr=\vecp/q\in\scrF_Q+\ZZ^{d-1}$, we then have 
\begin{equation}
	\gamma^{-1} \trans(n_-(\vecr)D(q))^{-1} = \bigg( \begin{pmatrix} q^{-1/(d-1)} 1_{d-1} & \trans\vecnull \\ \vecp & q \end{pmatrix}\gamma \bigg)^{-1} \in H .
\end{equation}
That is,
\begin{equation}\label{newlyn}
	\Gamma \trans(n_-(\vecr)D(q))^{-1}  \in \Gamma\backslash\Gamma H ,
\end{equation}
and thus, for $Q=\e^{(d-1)(t-\sigma)}$,
\begin{equation}\label{newlyn2}
\Gamma \trans(n_-(\vecr)\Phi^t)^{-1}\in \Gamma\backslash\Gamma H \{ \Phi^{-s} : s\in\RR_{\geq\sigma} \} .
\end{equation}
\end{remark}

\begin{lem}\label{embedLem}
The set $\Gamma\backslash\Gamma H \{ \Phi^{-s} : s\in\RR_{\geq\sigma} \}$ is a closed embedded submanifold of $\GamG$. 
\end{lem}

\begin{proof}
The set 
\begin{equation}\label{theset}
	\Gamma\backslash\Gamma H \{ \Phi^{-s} : s\in\RR_{\geq \sigma} \} =
	\Gamma\backslash\Gamma H \{ D(y_d) : y_d\in (0,c] \}, \qquad c=\e^{-(d-1)\sigma},
\end{equation}
is the image of the immersion map
\begin{equation}
	i : \scrH_0 \to \GamG, \qquad \Gamma_H M \mapsto \Gamma M,
\end{equation}
\begin{equation}
	\scrH_0 := \GamH \{ D(y_d) : y_d\in (0,c] \},
\end{equation}
and is thus an immersed submanifold of $\GamG$. To show that it is in fact a closed embedded submanifold, we need to establish that $i$ is a proper map, i.e., every compact $\scrK\subset\GamG$ has a compact pre-image $i^{-1}(\scrK)$; see e.g. \cite[Chapter III]{Boothby86}. Since $i$ is continuous, $i^{-1}(\scrK)$ is closed. It therefore suffices to show that $i^{-1}(\scrK)$ is contained in a compact subset of $\scrH_0$. 

For $M\in G$, let $I(M)=\inf\{\| \vecm M \|:\vecm\in\ZZ^d\setminus\{\vecnull\}\}$.
By Mahler's criterion, there is $\theta>0$ such that $I(M)\geq \theta$ for all $M\in G$ with $\Gamma M\in\scrK$. 
If $\Gamma_H M \in i^{-1}(\scrK)$, then $I(M)\geq\theta$ with $M=h D(y_d)$, $h\in H$. Thus $(\vecnull,1) M=y_d$ and therefore $y_d\geq\theta$. This implies that, for any $h\in H$,
\begin{equation}
	i(\Gamma_H h)=\Gamma h
	\in \scrK' := \scrK \{ D(y_d)^{-1} : \theta \leq y_d \leq c \} ,
\end{equation}
where $\scrK'$ is a compact subset of $\GamG$. 

It is a basic fact that, since $H$ is a closed subgroup of $G$ and $\Gamma_H=\Gamma\cap H$ is a lattice in $H$, the set $\Gamma\backslash\Gamma H$ is a closed embedded submanifold of $\GamG$ \cite[Theorem 1.13]{Raghunathan72}. We denote by $j:\GamH\to\Gamma\backslash\Gamma H$ the immersion map. Thus $j^{-1}(\scrK')$ is a compact subset of $\GamH$, and $i^{-1}(\scrK)$ is contained in the compact subset $j^{-1}(\scrK') \{ D(y_d) : \theta \leq y_d \leq c \}$ of $\scrH_0$.
\end{proof}

The significance of \eqref{newlyn2} and Lemma \ref{embedLem} is that it allows us reduce the continuity hypotheses of Theorem \ref{equiThm1} and Remark \ref{remB} to continuity of $\widetilde f$ restricted to the closed embedded submanifold 
\begin{equation}
	\TT^{d-1}\times\Gamma\backslash\Gamma H  \{ \Phi^{-s} : s\in\RR_{\geq\sigma} \} .
\end{equation}
We will exploit this fact in the proof of Theorem \ref{equiThm2b}.

\section{A variant of Theorem \ref{equiThm1} \label{secVariant}}

The following variant of Theorem \ref{equiThm1} will be key in the proof of  Theorem \ref{mainThm}. Recall the definition of $\uveca$ and $D(T)$ in \eqref{uveca} and \eqref{Ddef}, respectively.

\begin{thm}\label{equiThm2}
Let $\scrD\subset\{ \vecx\in\RR^d: 0<x_1,\ldots,x_{d-1}\leq x_d\}$ be bounded with boundary of Lebesgue measure zero, and $f:\overline\scrD\times\GamG\to\RR$ bounded continuous. Then
\begin{equation}\label{eqThm2}
	\lim_{T\to\infty} \frac{1}{T^d} \sum_{\veca\in\hatZZ^d\cap T\scrD} f\bigg(\frac{\veca}{T}, n_-(\uveca)D(T)\bigg)  
	= \frac{1}{\zeta(d)} \int_{\scrD\times\GamH}\widetilde f\big(\vecy,M D(y_d)\big) \, d\vecy \, d\mu_H(M)  
\end{equation}
with $\widetilde f(\vecx,M):=f(\vecx,\trans M^{-1})$.
\end{thm}

\begin{proof}
Let $g:\RR^{d-1}\times\GamG\to\RR$ be a bounded continuous function.
We apply Theorem \ref{equiThm1} with $T=\e^{(d-1)t}$, $c=\e^{-(d-1)\sigma}$, and the test function
\begin{equation}
	f(\vecx,M)= \sum_{\vecn\in\ZZ^{d-1}} g(\vecx+\vecn,M) \chi_{[0,1]^{d-1}}(\vecx+\vecn).
\end{equation}
Note that this sum has at most $2^{d-1}$ non-zero terms.
The function $f(\vecx,M)$ is bounded everywhere, and continuous on $[(0,1)^{d-1}+\ZZ^{d-1}]\times\GamG$; hence Remark \ref{remB}, together with the asymptotics \eqref{asymQ}, yield
\begin{equation}
	\begin{split}
	\lim_{T\to\infty} & \frac{\zeta(d)}{T^d} \sum_{\substack{\veca\in\hatZZ^d\\ 1\leq a_1,\ldots,a_{d-1} \leq a_d \\ a_d\leq c T}} g\big(\uveca,n_-(\uveca)D(T)\big)  \\
	& = (d-1)\int_{\sigma}^\infty \int_{[0,1]^{d-1}\times\GamH} \widetilde g(\vecx, M \Phi^{-s}) \, d\vecx \, d\mu_H(M) \, \e^{-d(d-1)s} ds 
	\\
	& = \int_0^c \int_{[0,1]^{d-1}\times\GamH} \widetilde g\big(\vecx, M D(y_d)\big) \, d\vecx \, d\mu_H(M) \, y_d^{d-1} dy_d 
\end{split}
\end{equation}
where we have substituted in the last step $y_d=\e^{-(d-1)s}$. So for any $0\leq b< c$ we have
\begin{equation}
	\begin{split}
	\lim_{T\to\infty} & \frac{1}{T^d} \sum_{\substack{\veca\in\hatZZ^d\\ 1\leq a_1,\ldots,a_{d-1}\leq a_d \\ bT< a_d\leq c T}} g\big(\uveca,n_-(\uveca)D(T)\big)  \\
	& = \frac{1}{\zeta(d)}\int_b^c \int_{[0,1]^{d-1}\times\GamH} \widetilde g\big(\vecx, M D(y_d)\big) \, d\vecx \, d\mu_H(M) \, y_d^{d-1} dy_d ,
\end{split}
\end{equation}
and hence for $h:\RR^{d-1}\times\RR\times\GamG\to\RR$ continuous with support in $\RR^{d-1}\times\scrI\times\GamG$ and $\scrI\subset\RR_{\geq 0}$ bounded, we have
\begin{equation}\label{still}
	\begin{split}
	\lim_{T\to\infty} & \frac{1}{T^d} \sum_{\substack{\veca\in\hatZZ^d\\ 1\leq a_1,\ldots,a_{d-1}\leq a_d}} h\bigg(\uveca,\frac{a_d}{T},n_-(\uveca)D(T)\bigg)  \\
	& = \frac{1}{\zeta(d)}\int_{[0,1]^{d-1}\times\scrI\times\GamH} \widetilde h\big(\vecx,y_d, M D(y_d)\big) \, d\vecx  \, y_d^{d-1} dy_d \, d\mu_H(M).
\end{split}
\end{equation}
We now take $h(\vecx,y_d, M)=\chi_{\scrD}(\vecx y_d,y_d)\,f((\vecx y_d,y_d), M)$ with $f$ as in Theorem \ref{equiThm2}, and substitute $\vecy'=\vecx y_d$. Note that with this choice $h$ is no longer continuous; but $\scrD$ has boundary of measure zero and thus Remark \ref{remB} applies. 
\end{proof}

Remark \ref{remy} and Theorem \ref{equiThm2} now imply the following theorem. Given a bounded subset $\scrD\subset\RR_{\geq 0}^d$, define
\begin{equation}
	\scrM_\scrD = \big\{ (\vecy,\Gamma \trans M^{-1} D(y_d)^{-1} ): (\vecy,\Gamma M)\in\overline\scrD\times \Gamma\backslash\Gamma H\big\} ,
\end{equation}
which, in view of Lemma \ref{embedLem}, is a closed embedded submanifold of $\RR^d\times\GamG$. 
The bijection 
\begin{equation}\label{bije}
	\overline\scrD\times\Gamma_H\backslash H \to \scrM_\scrD, \qquad(\vecy,\Gamma_H M )\mapsto (\vecy,\Gamma \trans M^{-1} D(y_d)^{-1} ),
\end{equation}
allows us to define a natural measure $\nu$ on $\scrM_\scrD$ as the pushforward of $\vol\times\mu_H$, where $\vol$ is Lebesgue measure on $\RR^d$ and $\mu_H$ as defined in \eqref{siegel2}.
In the following we understand the interior and closure of subsets of $\scrM_\scrD$ with respect to the topology of $\scrM_\scrD$.

Since $n_-(\uveca)D(T)=n_-(\uveca)D(a_d)D(a_d/T)^{-1}$,
eq.~\eqref{newlyn} implies that
\begin{equation}\label{mini}
	\bigg(\frac{\veca}{T}, \Gamma n_-(\uveca)D(T)\bigg) \subset\scrM_\scrD.
\end{equation}

\begin{thm}\label{equiThm2b}
Let $\scrD\subset\{ \vecx\in\RR^d: 0<x_1,\ldots,x_{d-1}\leq x_d\}$ be bounded with boundary of Lebesgue measure zero, and $\scrA\subset\scrM_\scrD$. Then
\begin{equation}\label{eqThm2b}
	\liminf_{T\to\infty} \frac{1}{T^d} \#\bigg\{ \veca\in\hatZZ^d : \bigg(\frac{\veca}{T}, \Gamma n_-(\uveca)D(T)\bigg)\in\scrA\bigg\} 
	\geq \frac{\nu\big(\scrA^\circ \big)}{\zeta(d)}   
\end{equation}
and
\begin{equation}\label{eqThm2bb}
	\limsup_{T\to\infty} \frac{1}{T^d} \#\bigg\{ \veca\in\hatZZ^d : \bigg(\frac{\veca}{T}, \Gamma n_-(\uveca)D(T)\bigg)\in\scrA\bigg\}  
	\leq \frac{\nu\big(\overline{\scrA} \big)}{\zeta(d)}   .
\end{equation}
\end{thm}

\begin{proof}
The inclusion \eqref{mini} shows that the limit relation \eqref{eqThm2} in Theorem \ref{equiThm2} holds for any bounded continuous function $f: \scrM_\scrD \to \RR$. We can thus once more apply the above probabilistic argument \cite[Chapter III]{Shiryaev} (used in the justification of Theorem \ref{equiThm0b}) to prove \eqref{eqThm2b} and \eqref{eqThm2bb}.
\end{proof}

\section{Upper and lower limits \label{secUpp}}

Let us first of all note that we may assume in Theorem \ref{mainThm} without loss of generality that $\scrD\subset [0,1]^d$. Secondly, due to the symmetry of $F(\veca)$ under any permutation of the coefficients $a_i$, we may assume that $\scrD\subset\{\vecx\in\RR^d: 0\leq x_1,\ldots,x_{d-1} \leq x_d \}$.
Thirdly, it is sufficient to prove Theorem \ref{mainThm} for all bounded subsets of $\{\vecx\in\RR^d: \eta \leq x_1,\ldots,x_{d-1} \leq x_d \}$,
for any fixed $\eta>0$. This is due to the fact that for any bounded set $\scrD\subset [0,1]^d$ with boundary of measure zero, 
\begin{equation}
	\lim_{T\to\infty} \frac{1}{T^d} \#\big\{ \veca\in\hatZZ^d\cap T\big(\scrD\setminus\RR_{\geq\eta}^d\big) \big\}
	=\frac{\vol\big(\scrD\setminus\RR_{\geq\eta}^d\big)}{\zeta(d)} \leq 
	\frac{d\, \eta}{\zeta(d)} .
\end{equation}

We will therefore assume in the remainder of this section that, in addition to the assumptions of Theorem \ref{mainThm},
\begin{equation}\label{DDT}
	\scrD\subset \{\vecx\in\RR^d: \eta \leq x_1,\ldots,x_{d-1} \leq x_d \leq 1\} ,
\end{equation}
for arbitrary fixed $\eta>0$.

The following is an immediate corollary of Theorem \ref{WThm} (set $T=\e^{(d-1)t}$ and recall that $W_\delta(\lambda \vecalf,M) = \lambda W_\delta(\vecalf,M)$ for any $\lambda>0$).

\begin{lem}\label{lamel}
Let $\veca\in\hatZZ_{\geq 2}^d\cap T\scrD$ with $\scrD$ as in \eqref{DDT}, and $0<\delta\leq \frac12$. Then
\begin{equation}
	\bigg| \frac{F(\veca)}{(a_1\cdots a_d)^{1/(d-1)}} 
	-  \frac{W_\delta(\vecy', n_-(\uveca)D(T))}{(y_1\cdots y_d)^{1/(d-1)}} \bigg| \leq \frac{d}{\eta\, T^{1/(d-1)}},
\end{equation}
where $\vecy=T^{-1}\veca$.
\end{lem}

In view of this lemma, the plan is thus to apply Theorem \ref{equiThm2b} with the set
\begin{equation}\label{fff}
	\scrA=\scrA_R=\bigg\{ (\vecy,\Gamma \trans M^{-1} D(y_d)^{-1}): \vecy\in\scrD,\; M\in\Gamma\backslash\Gamma H,\; \frac{W_\delta(\vecy', \trans M^{-1} D(y_d)^{-1})}{(y_1\cdots y_d)^{1/(d-1)}} > R \bigg\} .
\end{equation}
In the following, let 
\begin{equation}\label{MandM}
	M=\begin{pmatrix} A  & \trans\vecb \\ \vecnull & 1 \end{pmatrix} 
	\in H   ,
\end{equation}
where $A\in G_0$, $\vecb\in\RR^{d-1}$. Then
\begin{equation}
	\trans M^{-1}=\begin{pmatrix} \trans A^{-1}  & \trans\vecnull \\ -\vecb \trans A^{-1} & 1 \end{pmatrix}  ,
\end{equation}
and
\begin{equation}
	(\vecm+\vecxi)\trans M^{-1} D(y_d)^{-1} = \big( \big(\vecm'+\vecxi'-(m_d+\xi_d)\vecb \big) \trans A^{-1} y_d^{1/(d-1)}, (m_d+\xi_d) y_d^{-1} \big).
\end{equation}
Assuming $\xi_d\in(-\frac12,\frac12]$, we deduce that, for all $0<\delta\leq \frac12$, the statement
$(m_d+\xi_d)y_d^{-1}\in(-\delta,\delta)$ implies $m_d=0$ since $0<y_d\leq 1$. Therefore,
\begin{equation}\label{fdef2a}
\begin{split}
	& W_\delta(\vecalf,\trans M^{-1}D(y_d)^{-1}) \\
	& = \sup_{\vecxi\in\TT^d} \minplu\big\{ (\vecm+\vecxi)\trans M^{-1}D(y_d)^{-1} \cdot(\vecalf,0) :
	\vecm\in\ZZ^{d},\; (\vecm+\vecxi) \trans M^{-1} D(y_d)^{-1} \in \scrR_\delta \big\} \\
	& = y_d^{1/(d-1)} \sup_{\substack{\vecxi'\in\TT^{d-1}\\ \xi_d\in(-\delta y_d,\delta y_d)}} \minplu \big\{ (\vecm'+\vecxi'-\xi_d\vecb)\trans A^{-1} \cdot \vecalf :  \vecm'\in\ZZ^{d-1},\; (\vecm'+\vecxi'-\xi_d\vecb)\trans A^{-1}\in\RR_{\geq 0}^{d-1}\big\} .
\end{split}
\end{equation}
The substitution $\vecxi'\mapsto \vecxi'+\xi_d\vecb$ explains that the above supremum is independent of $\vecb$. So
\begin{equation}\label{fdef2b}
\begin{split}
	& W_\delta(\vecalf,\trans M^{-1}D(y_d)^{-1}) \\
	& = y_d^{1/(d-1)} \sup_{\vecxi'\in\TT^{d-1}} \minplu \big\{ (\vecm'+\vecxi')\trans A^{-1} \cdot \vecalf : \vecm'\in\ZZ^{d-1},\; (\vecm'+\vecxi')\trans A^{-1}\in\RR_{\geq 0}^{d-1}\big\} \\
	& =y_d^{1/(d-1)} V(\vecalf,\trans A^{-1}) ,
\end{split}
\end{equation}
where
\begin{equation}
\begin{split}
	V(\vecalf,A) & =\sup_{\veczeta\in\TT^{d-1}} \minplu \big\{ (\vecn+\veczeta) A \cdot \vecalf : \vecn\in\ZZ^{d-1},\; (\vecn+\veczeta)A\in\RR_{\geq 0}^{d-1}\big\} \\
	& = \sup_{\veczeta\in\TT^{d-1}} \minplu \big( (\ZZ^{d-1}+\veczeta) A\cap \RR_{\geq 0}^{d-1}\big) \cdot\vecalf .
\end{split}
\end{equation}
Now set $\vecalf=\vecy'=(y_1,\ldots,y_{d-1})$, and 
\begin{equation}
Y=(y_1\cdots y_{d-1})^{-1/(d-1)}\diag(y_1,\ldots,y_{d-1})\in G_0,
\end{equation}
so that $\vecy'=(y_1\cdots y_{d-1})^{1/(d-1)} \vece Y$. Then
\begin{equation}
	V(\vecy',A) = (y_1\cdots y_{d-1})^{1/(d-1)} V(\vece,A Y)
\end{equation}
and hence
\begin{equation}
	\frac{W_\delta(\vecy',\trans M^{-1}D(y_d)^{-1})}{(y_1\cdots y_{d})^{1/(d-1)}} =  V(\vece,\trans A^{-1} Y) .
\end{equation}
Set 
\begin{equation}
	V(A):=V(\vece,A)= \sup_{\veczeta\in\TT^{d-1}} \min \big( (\ZZ^{d-1}+\veczeta) A\cap \RR_{\geq 0}^{d-1}\big) \cdot\vece.
\end{equation}
We conclude that
\begin{equation}\label{fff222}
	\scrA_R=\bigg\{ \big(\vecy, \Gamma\trans M^{-1}D(y_d)^{-1}\big): (\vecy,M)\in\scrD\times\Gamma\backslash \Gamma H,\; V(\trans A^{-1} Y) > R \bigg\}.
\end{equation}

\begin{lem}
$V(A)$ is a continuous function on $\GamGG$.
\end{lem}

\begin{proof}
We have $V(\gamma A)=V(A)$ for all $\gamma\in\Gamma_0$ by the same argument as in \eqref{Winv}, and hence $V(A)$ is a function on $\GamGG$. It is sufficient to establish the continuity of $V(A)$ on compact subsets of $G_0$. Let us thus fix a compact set $\scrC\subset G_0$, and define
\begin{equation}
	K=\{ \veczeta A: \veczeta\in[0,1]^{d-1}, \; A\in\scrC \},
\end{equation}
which is a compact subset of $\RR^{d-1}$.
Then, for all $A\in\scrC$,
\begin{equation}
	V(A) = \sup_{\vecx\in L} \min \big( (\ZZ^{d-1}A+\vecx) \cap \RR_{\geq 0}^{d-1}\big) \cdot\vece,
\end{equation}
where $L$ is any set containing $K$.
Clearly $V(A)$ is bounded on $\scrC$, i.e., there is $R>0$ such that $V(A)\leq R$ for all $A\in\scrC$. Thus
\begin{equation}
	V(A) = \sup_{\vecx\in L} \min \big( (\ZZ^{d-1}A+\vecx) \cap R\Delta \big) \cdot\vece,
\end{equation}
where $\Delta$ is the simplex \eqref{simplex}. For $K'=K+[-1,1]\vece$,
\begin{equation}
	S=\ZZ^{d-1} \cap \bigcup_{A\in\scrC}\bigcup_{\vecx\in K'} \big((R\Delta-\vecx)A^{-1} \big)
\end{equation}
is a finite subset of $\ZZ^{d-1}$, and we have
\begin{equation}
	V(A) = \sup_{\vecx\in K'} \min_{\vecm\in S} \big( (\vecm A+\vecx) \cap \RR_{\geq 0}^{d-1} \big) \cdot\vece
\end{equation}
for all $A\in\scrC$. (The reason why we use $K'$ rather than $K$ in the definition of $S$ will become clear below.)

Fix $\epsilon\in(0,1)$. Then there exists $\delta>0$ such that, for all $A,A'\in\scrC$ with $d(A,A')<\delta$, we have 
\begin{equation}\label{AA}
	\|\vecm A-\vecm A'\|<\epsilon \qquad \text{for all $\vecm\in S$.}
\end{equation}
Thus, for any $\vecm\in S$ we have
\begin{equation} \label{first}
	\vecm A'+\vecx-\epsilon\vece\in\RR_{\geq 0}^{d-1} \quad \text{implies} \quad \vecm A+\vecx\in\RR_{\geq 0}^{d-1} , 
\end{equation}
and secondly
\begin{equation}\label{then}
\begin{split}
	(\vecm A'+\vecx-\epsilon\vece)\cdot\vece & =  (\vecm A'+\vecx)\cdot\vece - d\epsilon\\
	& \geq (\vecm A+\vecx)\cdot\vece - (\sqrt d+ d)\epsilon .
\end{split}
\end{equation}
Now choose $\vecx\in K$ such that
\begin{equation}
	\min \big( (\ZZ^{d-1} A+\vecx) \cap \RR_{\geq 0}^{d-1} \big) \cdot\vece \geq V(A) - \epsilon .
\end{equation}
Then \eqref{first} and \eqref{then} yield
\begin{equation}
	\min_{\vecm\in S} \big( (\vecm A'+\vecx-\epsilon\vece) \cap \RR_{\geq 0}^{d-1} \big) \cdot\vece
	\geq V(A) - (1+\sqrt d+ d) \epsilon.
\end{equation}
Since $\vecx-\epsilon\vece\in K'$ (because $\vecx\in K$ and $0<\epsilon<1$), the left hand side is at most $V(A')$. That is, $V(A')\geq V(A) - (1+\sqrt d+ d) \epsilon$. We conclude by interchanging $A$ and $A'$ that
\begin{equation}
	|V(A') - V(A)|\leq (1+\sqrt d+ d)\epsilon.
\end{equation}
for all $A,A'\in\scrC$ with $d(A,A')<\delta$.
\end{proof}

Since $V(A)$ is continuous, we have for any $\epsilon\in(0,R]$,
\begin{equation}
	\scrA_{R+\epsilon} \subset \scrA_R^\circ  \subset \overline\scrA_R \subset  \scrA_{R-\epsilon} .
\end{equation}

Define the  function $\Psi_d: \RR_{\geq 0} \to [0,1]$ by
\begin{equation}\label{limdi}
	\Psi_d(R):=\mu_0\big(\big\{ A\in \GamGG: V(A) > R \big\}\big),
\end{equation}
which is non-increasing. Note that by the invariance of $\mu_0$ under the right $G_0$-action and under $A\mapsto \trans A^{-1}$, we have 
\begin{equation}
	\Psi_d(R)=\mu_0\big(\big\{ A\in \GamGG: V(\trans A^{-1} Y) > R \big\}\big).
\end{equation}

As to the right hand sides of \eqref{eqThm2b} and \eqref{eqThm2bb}, the above calculations show that for any $\epsilon\in(0,R]$,
\begin{equation}
	\nu\big(\scrA_R^\circ\big) \geq \vol(\scrD)\,\Psi_d(R+\epsilon)
\end{equation}
and
\begin{equation}
	\nu\big(\overline\scrA_R\big) \leq \vol(\scrD)\,\Psi_d(R-\epsilon).
\end{equation}
Thus, combining these inequalities with Theorem \ref{equiThm2b} and Lemma \ref{lamel}, we obtain the following.

\begin{lem}
Let $R>0$. For any $\epsilon\in(0,R]$,
\begin{equation}
	\liminf_{T\to\infty} \frac{1}{T^d} \#\bigg\{ \veca\in\hatZZ_{\geq 2}^d\cap T\scrD : \frac{F(\veca)}{(a_1\cdots a_d)^{1/(d-1)}} >R \bigg\}
	\geq \frac{\vol(\scrD)}{\zeta(d)}\, \Psi_d(R+\epsilon) ,
\end{equation}
\begin{equation}
	\limsup_{T\to\infty} \frac{1}{T^d} \#\bigg\{ \veca\in\hatZZ_{\geq 2}^d\cap T\scrD : \frac{F(\veca)}{(a_1\cdots a_d)^{1/(d-1)}} >R \bigg\}
	\leq \frac{\vol(\scrD)}{\zeta(d)}\, \Psi_d(R-\epsilon) .
\end{equation}
\end{lem}

With this lemma, the proof of Theorem \ref{mainThm} is complete if we can show that $\Psi_d(R)$ is continuous (since then the $\limsup$ and $\liminf$ must coincide). This will be proved in Section \ref{secCon}.

\section{Lattice free domains and covering radii \label{secLat}}

We denote the standard basis vectors in $\RR^{d-1}$ by $\vece_1=(1,0,\ldots,0),\ldots,\vece_{d-1}=(0,\ldots,0,1)$.
Consider the simplex \eqref{simplex} and denote the face perpendicular to $\vece_i$ by $\Delta_i$ ($i=1,\ldots,d-1$), and by $\Delta_d$ the face perpendicular to $\vece$.

Recall from the previous section:
\begin{equation}
		V(A) = \sup_{\veczeta\in\TT^{d-1}} \min \big( (\ZZ^{d-1}+\veczeta) A\cap \RR_{\geq 0}^{d-1}\big) \cdot\vece .
\end{equation}

The following lemma states, that the simplex $\Delta$, enlarged by a factor of $V(A)$ and suitably translated, is a maximal lattice free domain; cf.~also \cite{Scarf93}.

\begin{lem}\label{lemSta}
If $V(A) = R$ for some $R>0$, then there is a vector $\veczeta\in\RR^{d-1}$ such that 
\begin{enumerate}
	\item[(i)] $\ZZ^{d-1}A\cap (R \Delta^\circ+\veczeta) = \emptyset$;
	\item[(ii)] $\ZZ^{d-1}A\cap (R\Delta_i^\circ+\veczeta) \neq \emptyset$ for all $i=1,\ldots,d$.
\end{enumerate}
On the other hand, if {\rm (i)} and {\rm (ii)} hold for some $R>0$, $\veczeta\in\RR^{d-1}$, then $R\leq V(A)$.
\end{lem}

\begin{proof}
If $\ZZ^{d-1}A\cap (R \Delta^\circ+\veczeta) \neq \emptyset$ for all $\veczeta$, then $V(A)<R$, contradicting our assumption $V(A)=R$. Hence there exists $\veczeta$ such that (i) holds.  If $\ZZ^{d-1}A\cap (R\Delta_i^\circ+\veczeta) = \emptyset$ for some $i$, then there exists a larger translate $R' \Delta^\circ+\veczeta'$ (for some $R'>R$, $\veczeta'\in\RR^{d-1}$) which is lattice free, and hence $V(A)\geq R' >R$. This proves (ii), and the final statement is evident.
\end{proof}

\begin{thm}
Denote by $\rho(A)$ the covering radius of the simplex $\Delta$ with respect to the lattice $\ZZ^{d-1}A$. Then
\begin{equation}
	\rho(A)=V(A).
\end{equation}
\end{thm}

\begin{proof}
(We adapt the argument of \cite[Theorem 2]{Scarf93}.) Let $V(A)=R$ and assume $\ZZ^{d-1}A+R\Delta\neq\RR^{d-1}$. Then there is $\vecxi\in\RR^{d-1}$ such that $\vecxi+\vecv\notin R\Delta$ for all $\vecv\in\ZZ^{d-1}A$. Hence $\ZZ^{d-1}A\cap (R\Delta-\vecxi) = \emptyset$, and, by Lemma \ref{lemSta}, $V(A)>R$; a contradiction. This shows $\rho(A)\leq V(A)$.

On the other hand, again by Lemma \ref{lemSta}, for any $R'<R=V(A)$ there exists $\veczeta\in\RR^{d-1}$ such that $\ZZ^{d-1}A\cap (R'\Delta+\veczeta) = \emptyset$, and hence no element of $\ZZ^{d-1}A$ is covered by the translates of $R'\Delta+\veczeta$. This proves $\rho(A)>R'$ and hence $\rho(A)=V(A)$.
\end{proof}

\section{Continuity of the limit distribution \label{secCon}}

The following lemma shows that $\Psi_d(R)$ is continuous.

\begin{lem}\label{0lem}
For every $R>0$,
\begin{equation}\label{zeromu}
	\mu_0 \big(\big\{ A\in\GamGG : V(A)=R \big\}\big) = 0.
\end{equation}
\end{lem}

\begin{proof}
By Lemma \ref{lemSta} (ii), the set $\big\{ A\in G_0 : V(A)=R \big\}$ is a subset of
\begin{equation}
	\bigcup_{\vecn_1,\ldots,\vecn_d\in\ZZ^{d-1}} \big\{ A\in G_0 : \text{there exists $\veczeta\in\RR^{d-1}$ such that } \vecn_i A \cap (R\Delta_i^\circ+\veczeta) \neq\emptyset  \;(i=1,\ldots,d) \big\}.
\end{equation}
We therefore need to show that each set in the above union has $\mu_0$-measure zero. Since the sets $R\Delta_i^\circ$ are contained in the respective hyperplanes $\vece_i\cdot\vecy=0$ (for $i=1,\ldots,d-1$) and $\vece\cdot\vecy=R$ (for $i=d$), it suffices to show that
\begin{multline}\label{jetset}
	\big\{ A\in G_0 :  \text{there exists $\veczeta=(\zeta_1,\ldots,\zeta_{d-1})\in\RR^{d-1}$ such that } \\ \vece_i\cdot\vecn_i A = \zeta_i \; (i=1,\ldots,d-1),\; \vece\cdot\vecn_d A= R+\vece\cdot\veczeta \big\} 
\end{multline}
has measure zero. Evidently \eqref{jetset} equals
\begin{equation}\label{jetset2}
\bigg\{ A\in G_0 :   \vece\cdot\vecn_d A= R + \sum_{i=1}^{d-1}\vece_i\cdot\vecn_i A \bigg\}
= \bigg\{ A\in G_0 :   \tr(LA)= R \bigg\},
\end{equation}
with the matrix
\begin{equation}
	L=\begin{pmatrix} \vecn_d-\vecn_1 \\ \vdots \\ \vecn_d-\vecn_{d-1} \end{pmatrix} .
\end{equation}
If $L=0$ the set \eqref{jetset2} is empty (since $R>0$) and hence has measure zero. If $L\neq 0$ then the set \eqref{jetset2} is a submanifold of codimension one; note that the map $G_0\to\RR$, $A\mapsto \tr(LA)$, has non-vanishing differential except at the (at most two) points $A\in G_0$ for which $LA$ is proportional to the identity matrix. Hence the set \eqref{jetset2} has measure zero also in this case and the proof is complete.
\end{proof}

\section*{Acknowledgements}

I thank Alex Gorodnik, Han Li, Andreas Str\"ombergsson and the referee for their comments on the first drafts of this paper. I am especially grateful to Andreas Str\"ombergsson for pointing out a gap in the proof of Lemma \ref{0lem}, and for providing a much simpler alternative. 

\begin{appendix}

\section{The distribution of sublattices}\label{AppA}

Sections \ref{secFarey} and \ref{secVariant} establish the equidistribution of Farey sequences embedded in large horospheres. These results provide an alternative perspective on Schmidt's work on the distribution of sublattices of $\ZZ^d$ \cite{Schmidt98}. In the present appendix, we will reformulate Theorems \ref{equiThm2} and \ref{equiThm2b} in a form that clarifies the relationship between the two approaches.

Let us fix a piecewise continuous map $K: \S_1^{d-1}\to G$ of the unit sphere $\S_1^{d-1}$ such that $\vecy K(\vecy)=(\vecnull,1)$. By {\em piecewise continuous} we mean here: there is a partition of $\S_1^{d-1}$ by subsets $\scrP_i$ with boundary of Lebesgue measure zero, so that $K$ restricted to $\scrP_i$ can be extended to a continuous map on the closure $\overline\scrP_i$.

We extend the definition of $K$ to $\RR^d\setminus\{\vecnull\}\to G$ by setting
\begin{equation}\label{K}
	K(\vecy)= K(\cvecy)D(\|\vecy\|)^{-1} 
\end{equation}
with $D$ as in \eqref{Ddef} and $\cvecy:=\vecy/\|\vecy\|$. The extended map still satisfies $\vecy K(\vecy)=(\vecnull,1)$. 

As in Remark \ref{remy}, we choose $\gamma\in\Gamma$ such that $\veca\gamma=(\vecnull,1)$. Then $(\vecnull,1)\gamma^{-1} K(\veca)=(\vecnull,1)$, which implies $\gamma^{-1} K(\veca)\in H$, and hence
$\Gamma K(\veca) \in \Gamma\backslash\Gamma H$.

\begin{thm}\label{equiThm2S}
Fix a piecewise continuous embedding $K: \RR^d\setminus\{\vecnull\}\to G$ as defined above. Let $\scrD\subset\RR^d$ be bounded with boundary of Lebesgue measure zero, and $f:\overline\scrD\times\Gamma\backslash\Gamma H\to\RR$ bounded continuous. Then
\begin{equation}\label{eqThm2S}
	\lim_{T\to\infty} \frac{1}{T^d} \sum_{\veca\in\hatZZ^d\cap T\scrD} f\bigg(\frac{\veca}{T}, K(\veca) \bigg)  
	= \frac{1}{\zeta(d)} \int_{\scrD\times\GamH} f\big(\vecy,M\big) \, d\vecy \, d\mu_H(M)  .
\end{equation}

\end{thm}

\begin{proof}
In view of the fact that $\Gamma\backslash\Gamma H$ is a closed embedded submanifold of $\GamG$, it suffices to prove that, for $f:\overline\scrD\times\GamG\to\RR$ bounded continuous,
\begin{equation}\label{eqThm2Sb}
	\lim_{T\to\infty} \frac{1}{T^d} \sum_{\veca\in\hatZZ^d\cap T\scrD} f\bigg(\frac{\veca}{T}, \trans K(\veca)^{-1} \bigg)  
	= \frac{1}{\zeta(d)} \int_{\scrD\times\GamH} \widetilde f\big(\vecy,M\big) \, d\vecy \, d\mu_H(M)  .
\end{equation}
We may assume without loss of generality that $f$ has compact support (cf.~ Step 0 of the proof of Theorem \ref{equiThm1}), and that $\scrD\subset\{\vecx\in\RR^d: \eta \leq x_1,\ldots,x_{d-1} \leq x_d \}\cap \RR_{> 0}\scrP_i$ for some fixed $\eta>0$ and $\scrP_i$ as defined in the second paragraph of this appendix.

If $\vecy\in\scrD$, then $y_d\geq \eta$, and we may expand
\begin{equation}
	K(\cvecy)^{-1}= \begin{pmatrix} A(\cvecy) & \trans\vecb(\cvecy) \\ \vecnull & 1 \end{pmatrix} M_{\cvecy} ,
\end{equation}
with $M_\vecy$ as in \eqref{My}. The maps $A$, $\vecb$ are continuous on $\overline\scrP_i\cap\RR_{\geq\eta}^d$, and hence bounded. A short calculation shows that
\begin{equation}
\begin{split}
	K(\vecy)^{-1} & = \begin{pmatrix} A(\cvecy) & \trans\vecb(\cvecy)\|\vecy\|^{-d/(d-1)} \\ \vecnull & 1 \end{pmatrix} M_{\vecy} \\
	& = \begin{pmatrix} 1_{d-1} & \trans\vecb(\cvecy)\|\vecy\|^{-d/(d-1)} \\ \vecnull & 1 \end{pmatrix} 
	\begin{pmatrix} A(\cvecy) & \trans\vecnull \\ \vecnull & 1 \end{pmatrix} M_{\vecy} .
\end{split}
\end{equation}
Set
\begin{equation}
	K_0(\vecy)^{-1}= \begin{pmatrix} A(\cvecy) & \trans\vecnull \\ \vecnull & 1 \end{pmatrix} M_{\vecy} .
\end{equation}	
Because $\|\veca\|\geq \sqrt{d}\, \eta T$, we have 
\begin{equation}
	d(\trans K(\veca)^{-1},\trans K_0(\veca)^{-1}) \leq \sup_{\vecy\in\scrD} \big\|\vecb(\cvecy)\big\|\; \big(\sqrt{d}\,\eta T\big)^{-d/(d-1)},
\end{equation}
where the supremum is finite by the continuity of $\vecb$.
Since $f$ is uniformly continuous, it therefore suffices to establish \eqref{eqThm2Sb} with $K(\veca)^{-1}$ replaced by $K_0(\veca)^{-1}$.
We now apply Theorem \ref{equiThm2} with the test function
\begin{equation}
	f_0(\vecy,M)=f\bigg(\vecy, M D(y_d) \begin{pmatrix} \trans A(\cvecy) & \trans\vecnull \\ \vecnull & 1 \end{pmatrix} \bigg) ,
\end{equation}
which is bounded continuous on $\overline\scrD\times \GamG$ (under the above assumptions on $f$ and $\scrD$). With this choice,
\begin{equation}
\begin{split}
	f_0\bigg(\frac{\veca}{T},n_-(\uveca)D(T) \bigg) & =f\bigg(\frac{\veca}{T}, n_-(\uveca)D(T) D(a_d/T) \begin{pmatrix} \trans A(\cveca) & \trans\vecnull \\ \vecnull & 1 \end{pmatrix} \bigg) \\
& =f\bigg(\frac{\veca}{T}, \trans K_0(\veca)^{-1} \bigg).
\end{split}
\end{equation}
As to the right hand side of \eqref{eqThm2}, we have
\begin{equation}
\begin{split}
	\widetilde f_0(\vecy,M D(y_d)) & = f_0\big(\vecy,\trans M^{-1} D(y_d)^{-1}\big)  \\
	& = f\bigg(\vecy, \trans M^{-1} D(y_d)^{-1} D(y_d) \begin{pmatrix} \trans A(\cvecy) & \trans\vecnull \\ \vecnull & 1 \end{pmatrix}\bigg) \\
	& = \widetilde f\bigg(\vecy, M \begin{pmatrix} A(\cvecy)^{-1} & \trans\vecnull \\ \vecnull & 1 \end{pmatrix}\bigg).
\end{split}
\end{equation}
Eq.~\eqref{eqThm2Sb} now follows from the right $H$-invariance of $\mu_H$.
\end{proof}

The following theorem is a corollary of Theorem \ref{equiThm2S}; the proof is analogous to that of Theorem \ref{equiThm2b}.

\begin{thm}\label{equiThm2bS}
Fix a piecewise continuous embedding $K: \RR^d\setminus\{\vecnull\}\to G$ as defined above. 
Let $\scrD\subset\RR^d$ be bounded with boundary of Lebesgue measure zero, and $\scrA\subset\overline\scrD\times\Gamma\backslash\Gamma H$. Then
\begin{equation}\label{eqThm2bS}
	\liminf_{T\to\infty} \frac{1}{T^d} \#\bigg\{ \veca\in\hatZZ^d : \bigg(\frac{\veca}{T}, \Gamma  K(\veca) \bigg)\in\scrA\bigg\} 
	\geq \frac{(\vol\times\mu_H)\big(\scrA^\circ \big)}{\zeta(d)}   
\end{equation}
and
\begin{equation}\label{eqThm2bbS}
	\limsup_{T\to\infty} \frac{1}{T^d} \#\bigg\{ \veca\in\hatZZ^d : \bigg(\frac{\veca}{T}, \Gamma K(\veca) \bigg)\in\scrA\bigg\}  
	\leq \frac{(\vol\times\mu_H)\big(\overline{\scrA} \big)}{\zeta(d)}   .
\end{equation}
\end{thm}

Let us now explain how the above statements are related to Schmidt's results on the distribution of primitive sublattices \cite{Schmidt98}. 

Two lattices $\Lambda,\Lambda'\subset\RR^d$ of rank $m$ are called {\em similar}, if there is an invertible angle-preserving linear transformation $R$ (that is, $R\in\RR_{>0}\OOO(d)$), such that $\Lambda'=\Lambda R$.

Let us denote by $\Gr_m(\RR^d)$ the Grassmannian of $m$-dimensional linear subspaces of $\RR^d$. The map
\begin{equation}
	\hatZZ^d \to \Gr_{d-1}(\RR^d),\qquad \veca \mapsto \veca^\perp :=\{ \vecx\in\RR^d: \vecx\cdot\veca=0 \}
\end{equation}
gives a one-to-one correspondence between primitive lattice points and rational subspaces of dimension $d-1$. 
A {\em primitive sublattice of $\ZZ^d$ of rank $d-1$} is defined as
\begin{equation}
	\Lambda_\veca = \ZZ^d \cap \veca^\perp ,
\end{equation}
and hence there is a one-to-one correspondence between primitive lattice points and primitive sublattices of rank $d-1$. The covolume of $\Lambda_\veca$ equals $\|\veca\|$. 
Note that 
\begin{equation}
	\veca^\perp \trans K(\veca)^{-1}= (\vecnull,1)^\perp=\RR^{d-1}\times\{0\} ,
\end{equation}
with $K(\veca)$ as in \eqref{K}. Hence
\begin{equation}
	\Lambda_\veca\trans K(\veca)^{-1} =\ZZ^d \trans K(\veca)^{-1} \cap (\RR^{d-1}\times\{0\})
\end{equation}
and
\begin{equation}\label{simmer}
	\Lambda_\veca\trans K(\veca)^{-1} =\|\veca\|^{-1/(d-1)} \Lambda_\veca \trans K(\cveca)^{-1}.
\end{equation}
We now choose the above embedding $K$ such that $K(\cvecy)\in\SO(d)$; see e.g. \cite[Section 4.2, footnote 3]{partI} for an explicit construction. 
The map
\begin{equation}
	\Lambda_\veca \mapsto \Lambda_\veca':=\Lambda_\veca\trans K(\veca)^{-1}
\end{equation}
maps primitive sublattices of $\ZZ^d$ of rank $d-1$ to lattices in $\RR^{d-1}$.
Eq.~\eqref{simmer} shows $\Lambda_\veca$ and $\Lambda_\veca'$ are similar; it furthermore implies that $\Lambda_\veca'$ has covolume one. 

In \cite{Schmidt98} Schmidt proves that, as $T\to\infty$, the set $\{\Lambda_\veca' : \|\veca\|\leq T\}$ becomes uniformly distributed in the space of lattices of covolume one,  $\GamGG$, with respect to the right $G_0$-invariant measure $\mu_0$. 
In particular, Theorem 3 in \cite{Schmidt98} (adapted to the case of primitive lattices of rank $d-1$) follows from our Theorem \ref{equiThm2bS}, if we set
\begin{equation}
	\scrA = \bigg\{ \bigg(\vecy,\Gamma \begin{pmatrix} A & \trans\vecb \\ \vecnull & 1 \end{pmatrix}\bigg) :  \vecy\in\scrD,\; A \in \scrA_0,\;\vecb\in\RR^{d-1} \bigg\}
	\subset\overline\scrD\times\Gamma\backslash\Gamma H,
\end{equation}
where $\scrD\subset\RR^d$ has boundary of Lebesgue measure zero, and $\scrA_0\subset\GamGG$ is arbitrary. Theorem 2 in \cite{Schmidt98} is obtained when $\scrD$ is taken to be the unit ball.

\end{appendix}

\end{document}